\documentclass[a4paper,11pt,makeidx]{amsart}
\usepackage{amscd}
\usepackage{xypic}  
\usepackage{amssymb}
\usepackage{amsthm}
\usepackage{epsf}
\makeindex

\newtheorem{thm}{Theorem}[section]
\newtheorem{cor}[thm]{Corollary}
\newtheorem{example}[thm]{Example}
\newtheorem{lemma}[thm]{Lemma}

\newtheorem{prop}[thm]{Proposition}
\newtheorem{rem}[thm]{Remark}

\newtheorem{defn}[thm]{Definition}
\newtheorem{thevarthm}[thm]{\varthmname}

\newenvironment{varthm*}[1]{\trivlist\item[]{\bf #1.}\it}{\endtrivlist}

\def\Hilb{\operatorname{Hilb}}

\def\Sing{\operatorname{Sing}}
\def\mult{\operatorname{mult}}

\def\max{\operatorname{max}}
\def\length{\operatorname{length}}
\def\c1{\operatorname{c_1}}
\def\c2{\operatorname{c_2}}

\def\gon{\operatorname{gon}}

\def\ZZ{{\mathbb Z}}

\def\PP{{\mathbb P}}

\def\C{{\mathcal C}}

\def\N{{\mathcal N}}
\def\O{{\mathcal O}}

\def\T{{\mathcal T}}

\def\*{\otimes}

\def\sub{\subseteq}

\def\+{\oplus}                   
\def\*{\otimes}                  
\def\khpil{\rightarrow}

\def\Supp{\operatorname{Supp}}

\def\mov{\operatorname{mov}}

\newcommand\calo{{\mathcal O}}
\newcommand\eps{\varepsilon}

\newcommand\lra{\longrightarrow}
\newcommand\beginproof[1]{\trivlist\item[\hskip\labelsep{\em #1.}]}
\newcommand\proofof[1]{\beginproof{Proof of #1}}
\def\endproof{\hspace*{\fill}\endproofsymbol\endtrivlist}

\def\endproofsymbol{\frame{\rule[0pt]{0pt}{6pt}\rule[0pt]{6pt}{0pt}}}

\begin{document}

\title{Moving curves and Seshadri constants}
\author{Andreas Leopold Knutsen, Wioletta Syzdek, Tomasz Szemberg}

\address{\hskip -.43cm Andreas Leopold Knutsen,
Department of Mathematics, University of Bergen,
Johannes Brunsgate 12, 5008 Bergen, Norway, e-mail {\tt andreas.knutsen@math.uib.no}}

\address{\hskip -.43cm Wioletta Syzdek,
   Mathematisches Institut, Universit\"at
   Duisburg-Essen, 45117 Essen, Germany,
e-mail {\tt wioletta.syzdek@uni-due.de}}

\address{\hskip -.43cm Tomasz Szemberg,
   Instytut Matematyczny AP,
   Podchora\.zych 2, PL-30-084 Krak\'ow, Poland,
   e-mail {\tt szemberg@ap.krakow.pl}}

\begin{abstract}
   We study families of curves covering a projective surface and
   give lower bounds on the self-intersection of the members of such
   families, improving results of Ein-Lazarsfeld and Xu. We apply the
   obtained inequalities to get new insights on Seshadri
   constants and geometry of surfaces.
\end{abstract}

\maketitle

\section{Introduction}
   Let $S$ be a smooth projective surface and
   $\C_U=\left\{C_u,x_u\right\}$ be a nontrivial family of irreducible
   pointed curves in $S$ such that $\mult_{x_u}C_u\geq m$ for some
   integer $m\geq 1$.

   Ein and Lazarsfeld showed in \cite{el}, that the self-intersection
   of each member of the family is bounded from below
   $$C_u^2\geq m(m-1).$$
   For $m=1$ this recovers a well known fact, that a curve with
   negative self-intersection cannot move in a family.

   For $m\geq 2$ Xu \cite[Lemma 1]{xu} gives a better bound of
   $$C_u^2\geq m(m-1)+1.$$

   Recall that the \emph{gonality} $\gon(X)$ of a smooth curve $X$ is
   defined as the minimal degree of a covering $X \lra \PP^1$. Our new bound is the following:

\begin{varthm*}{Theorem A}\label{thm:gonality}
   Let $S$ be a smooth projective surface.
   Suppose that $\C_U=\left\{C_u,x_u\right\}$ is a family
   of pointed curves as above
   parametrized by a $2$-dimensional subset $U\subset\Hilb(S)$
   and $C$ is a general member of this family. Let $\widetilde{C}$ be its normalization.
Then
   $$C^2\geq m(m-1)+\gon(\widetilde{C}).$$
\end{varthm*}
   The assumption that $U$ is $2$-dimensional is of course essential
   as there are surfaces fibred by curves of arbitrarily high
   gonality, which have self-intersection $0$. Note, that
   implicitly this assumption is made also in \cite[Lemma 1]{xu}, as
   a reduced curve is singular only in a finite number of points.

   An important point is that $\gon(\widetilde{C})$ can be bounded below in terms of the geometry of the surface. We will do this in Lemma \ref{lemma:help} below. Note for instance that clearly $\gon(\widetilde{C}) \geq 2$ if $S$ is nonrational, so that this already improves Xu's bound.

   Bounds as in Theorem A lead to interesting geometrical constrains
   on Seshadri constants on surfaces. We show in this direction the
   following result, which generalizes \cite[Theorem 3.2]{SyzSze07}.
   The precise definition of the integer $\mu_S$ appearing in the
   statement is given in \eqref{eq:defmu} in section \ref{sec:application}. For instance, $\mu_S \geq 2$ if $S$ is nonrational and $\mu_S \geq 3$ if $|K_S|$ is birational.
\begin{varthm*}{Theorem B} \label{main}
   Let $S$ be a smooth projective surface and $L$ a big and nef line
   bundle on $S$ such that for all $x\in S$
   \begin{equation} \label{inequality}
      \varepsilon(L;x) < \sqrt{(1-\frac{1}{4\mu_S})} \cdot  \sqrt{L^2}.
   \end{equation}
   Then $S$ is fibered by Seshadri curves.
\end{varthm*}

\section{Deformation theory and self-intersection of moving curves}

In this section we will prove Theorem A.

Let $C \subset S$ be a reduced and irreducible curve on a smooth surface. Let $p_a$ and $p_g$
denote the arithmetic and geometric genus, respectively, of $C$.

We have an exact sequence of coherent sheaves on $S$ (cf. \cite[(1.4)]{ser})
\[
\xymatrix{
 0 \ar[r] & \T_C \ar[r] & {\T_S}_{|C} \ar[r] & \N'_{C/S}  \ar[r] & 0,
}
\]
defining the {\it equisingular normal sheaf $\N'_{C/S}$ of $C$ in
$S$}. Its sections parametrize first-order deformations of $C$ in
$S$ that are locally trivial \cite[\S~4.7.1]{ser}, or  {\it
equisingular}.

Let $\widetilde{C}$ be the normalization of $C$ and $f:\widetilde{C} \khpil S$ the natural morphism.
Then we have a short exact sequence
\begin{equation} \label{eq:2}
\xymatrix{
 0 \ar[r] & \T_{\widetilde{C}} \ar[r]^{df}  & f^*\T_S \ar[r] & \N_f \ar[r] & 0,
}
\end{equation}
defining the {\it normal sheaf} $\N_f$ to $f$ (cf.
\cite[\S~3.4.3]{ser}). Its sections parametrize first-order
deformations of the morphism $f$, that is, first-order deformations
of $C$ that are {\it equigeneric} (i.e. of constant geometric genus).

Let $T \subset \N_f$ be the torsion subsheaf and $\overline{\N}_f:= \N_f/T$, which is locally free on ${\widetilde{C}}$.

On $C$ we also have a natural exact sequence
\[
\xymatrix{
 0 \ar[r] & \O_C \ar[r] & f_*\O_{\widetilde{C}} \ar[r] & \tau \ar[r] & 0,
}
\]
where $\tau$ is a torsion sheaf supported on $\Sing C$. This yields an exact sequence
$$
\xymatrix{
 0 \ar[r] &  \N'_{C/S} \ar[r] & f_* \overline{\N}_f \ar[r] & \N'_{C/S} \* \tau \ar[r] & 0
}
$$
(cf. \cite[(3.53)]{ser}). It follows that

\begin{equation}\label{eq:s3}
   h^0(\N'_{C/S}) \leq h^0(\overline{\N}_f) \leq h^0(\N_f).
\end{equation}
Geometrically, this means that the first-order equisingular
deformations are a subset of the first-order equigeneric
deformations.

   Now we are in the position to prove a slightly more precise
   version of Theorem A.
\begin{thm} \label{prop:intbound}
   Let $\C_U=\{C_u \ni x_u \}_{u \in U}$, $U \subset \Hilb(S)$,
   be a two-dimensional irreducible flat family of pointed, reduced and irreducible curves
   on a smooth projective surface $S$ such that $\mult_{x_u} C_u \geq m$ for all $u \in U$ and such that the gonality of
   the normalization of the general curve is $\ell$.
   Then
   \begin{equation} \label{eq:intbound}
     C_u^2\geq m(m-1)+\ell.
\end{equation}
Moreover, if equality holds, then, for general $u \in U$, $C_u$ is
smooth outside $x_u$, and has an ordinary $m$-tuple point at $x_u$.
\end{thm}

\begin{proof}
   Let $C$ be a general member of the family and $x$ the special point on $C$. Let $\widetilde{C}$ be the normalization
   of $C$ and $f:\widetilde{C} \khpil S$ the natural morphism.
   From \eqref{eq:2} and the definition of $\overline{\N}_f$ above we have
\begin{equation} \label{eq:5}
   p_g(C)=1+\frac12(K_S\cdot C+\deg(\overline{\N_f})+ \length T).
\end{equation}
Since a point of multiplicity $m$ causes the geometric genus of an irreducible curve to drop at least by $m \choose 2$ with respect
to the arithmetic genus, we must have
\begin{equation} \label{eq:6}
p_a(C) \geq {m \choose 2}+p_g(C)=\frac{1}{2}m(m-1)+p_g(C),
\end{equation}
so that
\begin{equation} \label{eq:6a}
C^2 =  2(p_a(C)-1)-K_S.C \geq m(m-1)+2(p_g(C)-1)-K_S.C.
\end{equation}
By our assumptions, $h^0(\N'_{C/S}) \geq 2$, so that
$h^0(\overline{\N}_f) \geq 2$ by \eqref{eq:s3}. Hence
$\deg(\overline{\N_f}) \geq \ell$. Combining
 \eqref{eq:5} and \eqref{eq:6a} we thus obtain
   $$C^2=2(p_a(C)-p_g(C))+\deg(\overline{\N}_f)\geq m(m-1)+\deg(\overline{\N}_f) \geq m(m-1)+\ell.$$
 This proves \eqref{eq:intbound}. If equality holds in \eqref{eq:intbound}, then clearly we must have
$\length T=0$ in \eqref{eq:5} and equality in \eqref{eq:6} and the last statement follows.
\end{proof}

In particular, we immediately see that $ C_u^2\geq m(m-1)+2$ if $S$
is non-rational.

\begin{rem} \label{rem:intbound}
{\rm If $K_S.C_u <0$, then from \eqref{eq:6a} one also obtains that either
  \begin{itemize}
  \item[(i)] $C_u^2 \geq m(m-1) - K_S.C_u$; or
  \item[(ii)] $C_u^2 =m(m-1) - K_S.C_u-2$, $p_g(C_u)=0$ (whence $S$ is rational) and $C_u$ is smooth outside $x_u$.
  \end{itemize}
In certain cases, as for instance in Example \ref{exa:cubic} below,
this may give a better bound than \eqref{eq:intbound}.}
\end{rem}

\section{Positivity of the canonical divisor and gonality}

The next result gives lower bounds on $\ell$ depending on the geometry of $S$.

We recall the following definition made in \cite{knman}:

\begin{defn} \label{def:birk}
  {\rm Let $L$ be a line bundle on a smooth projective variety $X$. Then $L$ is} birationally $k$-very ample
{\rm if there is a Zariski-open dense subset $U \sub X$ such that,
for any $0$-dimensional scheme $Z$ of length $k+1$ with $\Supp Z
\subset U$, the natural restriction map
\[ H^0(L) \rightarrow H^0(L \* \O_Z) \]
is surjective. }
\end{defn}

 For instance, $L$ is birationally $0$-very ample if and only if it has a section
 and $1$-very ample if and only if the rational map $\varphi_L$ determined by $|L|$ is birational
 onto its image.

We note that the notion of birational $k$-very ampleness is the ``birational version'' of the ordinary notion of $k$-very ampleness, in the sense that if $X' \rightarrow X$ is a birational morphism between
smooth projective varieties and $L$ is a line bundle on $X$, then $L$ is birationally $k$-very ample if and only if $f^*L$ is.

\begin{lemma} \label{lemma:help}
 Let $U \subset \Hilb(S)$
   be a reduced and irreducible scheme parametrizing a flat family of reduced and irreducible curves
   on a smooth projective surface $S$ such that  the gonality of
   the normalization of the general curve is $\ell$.
\begin{itemize}
\item[(a)] If $\dim U \geq 1$ and $K_S$ is birationally $k$-very ample, then $\ell \geq k+2$.
\item[(b)] If $\dim U \geq 2$ and $S$ is birational to a surface admitting a surjective morphism onto a smooth curve $B$ of gonality
$b$, then $\ell \geq b$.
\end{itemize}
\end{lemma}

\begin{proof}
To prove (a), pick a $1$-dimensional reduced and irreducible
subscheme $U'$ of $U$. After compactifying and resolving the
singularities of the universal family over $U'$, we obtain a smooth
surface $T$, fibered over a smooth curve, with general fiber $F$ a
smooth  curve of gonality $\ell$, and a surjective morphism $f:T \to
S$. By adjunction $K_T$ fails to be $(\ell-1)$-very ample on the
general fiber $F$. More precisely, on the general $F$, there is a
one-dimensional family of schemes $\{Z\}$ of length $\ell$ such that
the evaluation map
\[ H^0(K_T) \rightarrow H^0(K_T \* \O_Z) \]
is not surjective. Since $K_T = f^*K_S + R$, where $R$ is the
(effective) discriminant divisor of $f$, and $f$ is generically
$1:1$ on the fibers, we see that $K_S$ fails to be birationally
$(\ell-1)$-very ample. Hence, $k \leq \ell-2$.

As for (b), there is by assumption a birational morphism from a
smooth surface $\widetilde{S}$ to $S$ such that there is a
surjective morphism $\widetilde{S} \rightarrow B$.  Of course, also
$\widetilde{S}$ is dominated by a two-dimensional nontrivial family
of reduced and irreducible curves having normalizations of gonality
$\ell$. A general such curve must dominate $B$, whence so does its
normalization, so that $\ell \geq b$.
\end{proof}

\section{Examples}

We here give a few examples involving Theorem \ref{prop:intbound} and Lemma \ref{lemma:help}.

\begin{example} \label{exa:cubic}
(Cubic surfaces and $\PP^2$) {\rm We consider a smooth cubic surface
$S$ in $\PP^3$ and families of hyperplane sections satisfying the
conditions in Theorem \ref{prop:intbound}. Then $C_u.K_S=-3$ and
$p_a(C_u)=1$. Now Remark \ref{rem:intbound} yields that either
$C_u^2 \geq m(m-1) +3$ or $C_u^2 =m(m-1) +1$, $C_u$ is rational and
smooth outside $x_u$ and $x_u$ is a node.

The first case happens with $m=1$ for the family of all hyperplane sections and the second indeed happens for the $2$-dimensional family of tangent sections.

Similarly, for $S=\PP^2$, any $2$-dimensional family of lines
satisfies $C_u^2 =m(m-1) +1$ for $m=1$. }
\end{example}

\begin{example} \label{exa:abe}
  (Abelian surfaces)
  {\rm Let $S$ be a smooth abelian surface and $L$ a globally generated line bundle on $S$. For any family as in the hypotheses of Theorem \ref{prop:intbound}, we have that $C_u^2 \geq m(m-1)+2$. The following is an example where equality is attained.

  Let $S$ be an abelian surface with
  irreducible principal polarization $\Theta$. We assume that $\Theta$ is
  symmetric. Let $\mu$ denote the endomorphism
  $$\mu:\;S\ni x\longrightarrow 2\cdot x\in S$$
  and let $C=\mu^*(\Theta)$. Then it is well known (cf. \cite[Proposition II.3.6]{LB}),
  that $C\in|4\Theta|$ and $C$ has multiplicity $m=6$ at
  the origin. Of course $p_g(C)=2$, so that $\gon(C)=2$. We have
  $$32=C^2=6\cdot 5+2.$$
  Translates of $C$ give a two-dimensional family of curves which
  are actually algebraically equivalent to $C$ but not linearly
  equivalent.}
\end{example}

\begin{example} \label{exa:k3}
($K3$ surfaces) {\rm Let $S$ be a smooth $K3$ surface and $L$ a
globally generated line bundle on $S$. It is well-known that for any
positive integer $n$, there exist such $S$ and $L$ with $L^2=2n$.
Let $\{C_u\}_{u \in U}$ be any family of curves in $|L|$ satisfying
the conditions in Theorem \ref{prop:intbound}. Then  the theorem
yields that
\begin{equation} \label{eq:k3}
  L^2 \geq m(m-1)+2.
\end{equation}

If $m=1$, this yields that $L^2 \geq 2$, which is optimal, as $\dim |L|=\frac{1}{2}L^2+1=2$ if $L^2=2$.

If $m=2$, \eqref{eq:k3} yields that $L^2 \geq 4$, which is also
optimal. Indeed, if $L^2=4$, then $\dim |L|=3$ and for any $x \in
S$, we have $\dim |L \* \mathfrak{m}_x^2| \geq 3-3=0$, so that there
is a $2$-dimensional family of curves with a point of multiplicity
two. A general such curve is nodal of genus two, hence also of
gonality $2$.

If $m=3$, \eqref{eq:k3} yields that $L^2 \geq 8$, and we do not know
if this is optimal.
If $L^2=10$, however, then $\dim |L|=6$
and for any $x \in S$, we have $\dim |L \* \mathfrak{m}_x^3| \geq
6-6=0$, so that there is a $2$-dimensional family of curves with a
point of multiplicity three. }
\end{example}

\begin{example} \label{exa:gen}
   (Surfaces of general type) \rm Let $S$ be a smooth surface of degree
   $d \geq 5$ in $\PP^3$ and consider the family $C_x$ of
   tangent hyperplane sections of $S$, parametrized by $x \in S$.
   Then a general member of this family is irreducible and has
   multiplicity $m=2$ at $x$. On the other hand for the canonical
   divisor $K_S$ we have
   $$K_S=\calo_S(d-4),$$
   so that it is  $(d-4)$-very ample. This shows that equality can hold in \eqref{eq:intbound} in
   Theorem \ref{prop:intbound} even if $K_S$ is $k$-very
   ample and not merely birationally $k$-very ample, cf. Lemma \ref{lemma:help}.
\end{example}

\section{Applications to Seshadri constants}\label{sec:application}

   In this section we will prove Theorem B, as well as its
   consequences.

   We recall here only very basic notions connected to Seshadri
   constants. For a systematic introduction to this circle of ideas
   we refer to chapter 5 in Lazarsfeld's book \cite{PAG}.
\begin{defn}\rm
   Let $X$ be a smooth projective variety and $L$ a nef line bundle
   on $X$. For a fixed point $x\in X$ the real number
   $$\eps(L;x):=\inf\frac{L\cdot C}{\mult_x C}$$
   is the \emph{Seshadri constant of $L$ at $x$} (the
   infimum being taken over all irreducible curves $C$ passing through
   $x$).
\end{defn}
   For arbitrary line bundles Nakamaye introduced in \cite{Nak02} the
   notion of \emph{moving Seshadri constants} $\eps_{\mov}(L;1)$, see also \cite{RV}. If
   $L$ is nef, then the both notions coincide. If $L$ is big, then
   the function $\eps_{\mov}(L,-)$ on $S$ assumes its maximal value for
   $x$ very general, that is, away from a countable union of proper
   Zariski closed subsets. This was proved by Oguiso \cite{og} for $L$
   ample but the proof goes through for $L$ big. This maximal value
   will be denoted by $\eps(L;1)$. We use it in Corollary \ref{cor}.

   If there exists a curve $C\subset X$ actually computing the
   infimum in the above definition, then we call such a curve a
   \emph{Seshadri curve} (for $L$, through $x$).
   It is not known if Seshadri curves exist
   in general.

   If $X=S$ is a surface, then there is an upper bound
   on Seshadri constants
   $$   \eps(L;x)\leq\sqrt{L^2}\;.
   $$
   If $\eps(L;x)$ is \emph{strictly} less that $\sqrt{L^2}$, then
   there exists a Seshadri curve through $x$. This observation is
   fundamental for both results presented in this section.

   Before proving Theorem B, we need to introduce the following
   notation. This takes care of bounding the gonality of curves in
   the family appearing in the proof of Theorem B.
   We define
\begin{eqnarray}
  \label{eq:defmu}
 \mu_S & := & \max \Big\{ d \; | \; K_S \; \mbox{is birationally $(d-2)$-very ample or} \\
\nonumber       &    & S \; \mbox{is birational to a surface
dominating a smooth curve of gonality $d$} \Big  \}.
\end{eqnarray}

In particular, note that $\mu_S \geq 2$ if $S$ is non-rational and
$\mu_S \geq 3$ if $|K_S|$ defines a birational map.

\proofof{Theorem B}
   First of all \eqref{inequality} implies that
   for every point $P\in S$ there exists a Seshadri curve $C_P$.
   Such a curve need not be unique but there are only finitely
   many of them for every point $P$, see
   \cite[Prop. 1.8]{hab}. Since there are only countably many
   components in the Hilbert scheme of curves in $S$, there must
   exist a component containing at least a $1$-dimensional family $U$ of curves
   $C_P$. If there is more than one such component, we take one for
   which the dimension of $U$ is maximal.
   Note, that for curves in $U$ we must have $C_P^2\geq 0$.
   Let $m$ denote the multiplicity of a general member
   of this family in its distinguished point.

   It might happen that $\dim U=1$, but this means
   that a Seshadri curve $C_P$ equals to $C_Q$ for $Q\in C_P$ very
   general. Of course in this case it must be that $m=1$, because a Seshadri
   curve is reduced and irreducible. Then the index
   theorem together with \eqref{inequality} gives
   $$
      C_P^2L^2\leq (L.C_P)^2 < (1-\frac{1}{4\mu_S})L^2 < L^2,
   $$
   which implies $C_P^2=0$. It is well known that a moving curve of self-intersection $0$
   is semi-ample and we are in case b) of the Theorem.

   So we may assume that $\dim U=2$ and $m\geq 2$. Hence
   $C_P^2 \geq m(m-1)+\mu_S$ by Theorem \ref{prop:intbound} and Lemma \ref{lemma:help}.
   Revoking again the index theorem we therefore get
   \begin{equation} \label{eq:m2}
      (m(m-1)+\mu_S) L^2 \leq C_P^2L^2 \leq (C_P.L)^2 =m^2\varepsilon(L;1)^2 < m^2(1-\frac{1}{4\mu_S}) L^2 .
   \end{equation}
   It is elementary to observe that the real valued function
   $$
      f(m)=\frac{m(m-1)+ \mu_S}{m^2}
   $$
   for $m\geq 2$ has a minimum at $m=2\mu_S$ and
   $f(2\mu_S)=1-\frac{1}{4\mu_S}$, so that we arrive
   to a contradiction with \eqref{eq:m2}. This concludes the proof of
   Theorem B.
\endproof

   As a consequence, we obtain the following result yielding lower
   bounds on the Seshadri constant of the canonical bundle at very
   general points, which improves and generalizes
   \cite[Thms. 2 and 3]{bs}.

\begin{cor}\label{cor}
   Let $S$ be a minimal smooth projective surface of general type i.e. such that $K_S$ is big and nef. Then either
   $\varepsilon(K_S;1) \geq \sqrt{\frac{7}{8} K_S^2}$ or $S$ is fibered by Seshadri curves of $K_S$-degree
   $\varepsilon(K_S;1) \geq 2$.

   If furthermore $K_S$ is birationally $k$-very ample, for an integer $k \geq 1$, then either
   $\varepsilon(K_S;1) \geq \sqrt{(1-\frac{1}{4(k+2)}) K_S^2}$ or $S$ is fibered by Seshadri curves of $K_S$-degree
   $\varepsilon(K_S;1) \geq 4k$.
\end{cor}

   Note that the hypothesis that $S$ is minimal i.e.
   $K_S$ is nef is not restrictive, since if $\widetilde{S}
   \rightarrow S$ is a birational morphism, then one easily sees that
   $\varepsilon_{\mov}(K_{\widetilde{S}};1) \geq \varepsilon(K_S;1)$.

\proof
   If $K_S$ is big and nef, but not birationally $1$-very ample,
set $k=0$. Then, under the assumptions of the theorem, we have
$\mu_S \geq k+2$. By Theorem B, if $\eps(K_S;1) <
\sqrt{(1-\frac{1}{4(k+2)})K_S^2}$, then $S$ is fibered by (smooth)
Seshadri curves $C$ of $K_S$-degree $C.K_S = \eps(K_S;1)$. If $K_S$
is big and nef, we must have $C.K_S \geq 2$ by adjunction, as
$C^2=0$, and the first assertion follows.

If $K_S$ is birationally $1$-very ample for $k \geq 1$, then the general curve $C$ must have gonality
$\gon(C) \geq k+2$ by Lemma \ref{lemma:help}(a). Moreover, by adjunction we have $K_S.C=2p_a(C)-2$ and by Brill-Noether theory, $\gon(C) \leq \frac{p_a(C)+3}{2}$. Hence
\[ K_S.C = 2p_a(C)-2 \geq 2(2\gon(C)-3)-2 \geq 2(2(k+2)-3)-2 =4k, \]
and the second assertion follows.
\endproof

   Our next and last result parallels \cite[Theorem]{el} and \cite[Theorem 1]{xu}.
   The arguments are similar to those in the proof of Theorem B, but it seems that
   the result itself is of independt interest and is not a straightforward
   corollary of Theorem B.
   Note that for $\mu_S=1$, we retrieve (in practice) Xu's result.

\begin{prop} \label{thm:xuimpro}
  Let $S$ be a smooth surface and $L$ a big and nef line bundle on $S$. Assume that, for a given integer $a \geq 1$, we have that either
\[ L^2 > \frac{4\mu_Sa^2-4a+4}{4\mu_S-1}, \]
or
\[ L^2 = \frac{4\mu_Sa^2-4a+4}{4\mu_S-1} \; \; \mbox{and} \; \; \frac{L^2-2a}{2(L^2-a^2)} \not \in \ZZ. \]

Then, either
\begin{itemize}
\item[(i)] $\varepsilon(L;1) \geq a$, or
\item[(ii)] $S$ is fibered in Seshadri curves (of $L$-degree $< a$).
\end{itemize}
\end{prop}

\begin{rem} \label{rem:xuimpro}
  {\rm The case $\frac{L^2-2a}{2(L^2-a^2)} \not \in \ZZ$ also includes the case $L^2=a^2$. However, if $L^2 \geq  \frac{4\mu_Sa^2-4a+4}{4\mu_S-1}$, it is easily seen that $L^2 \geq a^2$, with equality only if $L^2=4$ and $a=2$. The proposition holds in this special case.}
\end{rem}

\begin{proof}
  As in the proof of \cite[Thm.]{el} or \cite[Theorem 1]{xu}, the set
\[ V:= \Big\{ (C,x) \; | \; x \in C \subset S \; \mbox{ a reduced, irreducible curve,} \; \mult_x(C) > \frac{C.L}{a} \Big\} \]
consists of at most countably many families. (This conclusion also holds when $L$ is big and nef but not ample.)

Assume that we are not in case (i) of the theorem, that is, assume that $\varepsilon(L;1) < a$. Then $\dim V \geq 2$.

By our assumptions we have $L^2 \geq a^2$, cf. Remark
\ref{rem:xuimpro}. Therefore, the index theorem yields, for any
curve $C \subset S$, that
\[ C.L \geq \sqrt{L^2 \cdot C^2} \geq a \sqrt{C^2}. \]
It follows that if $C.L <a$ for a curve $C \subset S$ moving in a nontrivial algebraic family, then $C^2=0$, so that we are in case (ii) of the theorem. Otherwise we must have a nonempty subset of
dimension $\geq 2$ of $V$ consisting of pairs $(C,x)$ with $\mult_x(C) > 1$. Since each curve in question is reduced, we can find a two-dimensional irreducible scheme $U \subset \Hilb S$ parametrizing curves in $V$ each with $\mult_{x} C >1$. Letting $m$ be the multiplicity of the general curve in this family and $C$ the algebraic equivalence class, we have $C^2 \geq m(m-1)+\mu_S$ by Proposition \ref{prop:intbound} and Lemma \ref{lemma:help}.

As $m > \frac{C.L}{a}$, we have $C.L \leq ma-1$, whence by the index
theorem,
\[ (ma-1)^2 \geq (C.L)^2 \geq C^2L^2 \geq (m^2-m+\mu_S)L^2, \]
so that
\begin{equation}
  \label{eq:meno}
  (L^2-a^2)m^2+(2a-L^2)m+(L^2\mu_S-1) \leq 0.
\end{equation}
If $L^2 \leq a^2$, then $L^2=a^2=4$ by Remark \ref{rem:xuimpro}, so that  \eqref{eq:meno} yields
$4\mu_S \leq 1$, a contradiction. Hence $L^2 > a^2$, so that the function
\[ f(x):= (L^2-a^2)x^2+(2a-L^2)x+(L^2\mu_S-1) \]
attains its minimal value $f_{min}$ at $x= \frac{L^2-2a}{2(L^2-a^2)}$ and
\[
  f_{min} = \frac{L^2}{4(L^2-a^2)} \Big[ (4\mu_S-1)L^2-4(\mu_Sa^2-a+1) \Big].
\]
By our assumptions on $L^2$ it follows that $f(x) \geq 0$ for all $x$ and that
$f(x) =0$ only if $L^2 = \frac{4\mu_Sa^2-4a+4}{4\mu_S-1}$ and then at the point
$x=\frac{L^2-2a}{2(L^2-a^2)}$, which is not an integer by assumption. Therefore
$f(m) >0$ for all $m \in \ZZ$, contradicting \eqref{eq:meno}.
\end{proof}

\paragraph*{{\it Acknowledgements.}}
This work has been supported by the SFB/TR 45 ``Periods, moduli
spaces and arithmetic of algebraic varieties''.
   The second and the third author were partially supported by
   a MNiSW grant N~N201 388834. The authors thank E.~Sernesi and F.~Flamini for helpful correspondence.

%
%


\begin{thebibliography}{B-F-H-V}


\bibitem[B-S]{bs} Th.~Bauer, T.~Szemberg,
{\it Seshadri constants on surfaces of general type},
Manuscr. Math. {\bf 126} (2008),167--175.

\bibitem[E-L]{el} L.~Ein, R.~Lazarsfeld, {\it Seshadri constants on smooth surfaces}.
Journ{\'e}es de G{\'e}om{\'e}trie Alg{\'e}brique d'Orsay (Orsay, 1992).  Ast{\'e}risque
{\bf 218}  (1993), 177--186.

\bibitem[RV]{RV}
    Ein, L., Lazarsfeld, R., Mustata, M., Nakamaye, M., Popa, M.:
    Restricted volumes and base loci of linear series
    arXiv:math/0607221

\bibitem[L-B]{LB}
H. Lange, Ch. Birkenhake, \textit{Complex Abelian Varieties},
Grundlehren der mathematischen Wissenschaften \textbf{302},
Springer-Verlag, 1992.

\bibitem[Knu]{knman}
   A.~L.~Knutsen:
\textit{On $k$th order embeddings of $K3$ surfaces and Enriques surfaces},
Manuscr. Math. \textbf{104} (2001), 211--237.

\bibitem[Nak]{Nak02}
   M. Nakamaye:
   \textit{Base loci of linear series are numerically determined},
   Trans. Amer. Math. Soc. \textbf{355} (2002), 551–-566.

\bibitem[PAG]{PAG}
   Lazarsfeld, R.:
   {\it Positivity in Algebraic Geometry~I}.
   Springer-Verlag, 2004.

\bibitem[Ogu]{og}
   K.~Oguiso:
   {\it Seshadri constants in a family of  surfaces},
   Math. Ann. {\bf 323} (2002), 625--631.

\bibitem[Ser]{ser} E.~Sernesi, {\it Deformations of algebraic schemes}. Grundlehren der Mathematischen Wissenschaften,
{\bf 334}. Springer-Verlag, Berlin, 2006.

\bibitem[S-S]{SyzSze07}
   Syzdek, W., Szemberg, T.:
   {\it Seshadri fibrations of algebraic surfaces},
   arXiv:0709.2592v1 [math.AG],
   to appear in: Math. Nachr.

\bibitem[Sze]{hab}
   T.~Szemberg, {\it Global and local positivity of line bundles},
   Habilitationsschrift, Essen 2001.

\bibitem[Xu]{xu} G.~Xu,
{\it Ample line bundles on smooth surfaces}, J. Reine Angew. Math.
{\bf 469} (1995), 199--209.

\end{thebibliography}
\end{document}